\documentclass[12pt,letterpaper,titlepage]{amsart}
\usepackage{amsmath, amssymb, amsthm, amsfonts,amscd,xr}

\usepackage{graphics}

\newcommand{\Z}{\mathbb{Z}}
\newcommand{\R}{\mathbb{R}}
\newcommand{\C}{\mathbb{C}}

\def\beq{\begin{equation}}
\def\eeq{\end{equation}}

\def\arr{\hbox to 20pt{\rightarrowfill}}

\def\Sl{\mathrm {Sl}} 
\def\PSl{\mathrm {PSl}} 

\def\SO{\mathrm {SO}} 
\def\SU{\mathrm {SU}}
 
\def\D{\mathcal D}

\def\W{\mathcal W}

\def\L{\mathcal L}

 \newenvironment{res} 
               {\begin{equation} 
\begin{minipage}{0.85\textwidth}} 
               { \end{minipage}\end{equation} } 

\def\ber{\begin{res} } 
\def\eer{\end{res}} 
 
\numberwithin{equation}{section} 
\newtheorem{thm}{Theorem}[section]

\newcommand{\oline}{\overline}

\newtheorem{lem}[thm]{Lemma}

\newtheorem{cor}[thm]{Corollary} 
 
\newtheorem{prop}[thm]{Proposition} 
\newtheorem{df}[thm]{Definition}

\makeatletter 
\def\section{\@startsection {section}{1}{\z@}{3.5ex plus 1ex minus 
    .2ex}{2.3ex plus .2ex}{\large\bf}} 
    \def\subsection{\@startsection{subsection}{2}{\z@}{3.25ex plus 1ex minus 
 .2ex}{1.5ex plus .2ex}{\bf}} 
\makeatother 
 
\def\bysame{\leavevmode\hbox to3em{\hrulefill}\,} 

\def\Ad{\operatorname{Ad}}

\def\vf{\mathfrak{v}}
\def\zf{\mathfrak{z}}
\def\af{\mathfrak{a}}
\def\sf{\mathfrak{s}} 
\def\Cr{\operatorname {Cr}}

\def\diag{\operatorname {diag}}

\def\gf{\mathfrak{g}}

\def\kf{\mathfrak{k}} 

\def\nf{\mathfrak{n}}

\def\id{\operatorname{id}}

\def\L{\mathcal{L}} 
 
\def\P{\mathbb{P}}

\def\W{\mathcal{W}}

\def\bs{\backslash}

\renewcommand{\Re}{\mbox{\rm Re}\,} 
\renewcommand{\Im}{\mbox{\rm Im}\,} 
\makeindex 

\begin{document} 
\title[Homogeneous harmonic spaces]
{The complex crown for homogeneous harmonic spaces}

\author{Roberto Camporesi  and Bernhard Kr\"otz}

\address{Dipartimento di Matematica\\
Politecnico di Torino\\
Corso Duca degli Abruzzi 24\\ 10129 Torino\\\
	 e-mail: camporesi@polito.it}

\address{Leibniz Universit\"at Hannover\\ Institut f\"ur Analysis\\
Welfengarten 1\\ D-30167 Hannover
\\email: kroetz@math.uni-hannover.de}

\date{\today} 
\thanks{}
\maketitle

\section{Introduction}

Let $X$ be a simply connected homogeneous harmonic Riemannian space. Then according to \cite{He}, 
Corollary 1.2, $X$ is isometric (up to scaling of the metric) to one of the following spaces:

\begin{enumerate}
\item $\R^n$.
\item $S^n$, $P^k(\C)$, $P^l(\mathbb{H})$, or $P^2(\mathbb{O})$, i.e., a compact rankone symmetric space. 
\item $H^n(\R)$, $H^k(\C)$, $H^l(\mathbb{H})$, or $H^2(\mathbb{O})$, i.e., a noncompact rankone symmetric space. 
\item a solvable Lie group $S=A \ltimes N$ where $N$ is of Heisenberg-type and 
$A\simeq \R^+$ acts on $N$ by anisotropic dilations preserving the grading.
\end{enumerate}
 
We denote by $\L$ the Laplace-Beltrami operator on $X$. In case (i),  $\L$-eigenfunctions on $\R^n$ 
extend to holomorphic functions on $\C^n$.  Likewise, in case (ii), $\L$-eigenfunction on $X=U/K$ 
admit holomorphic continuation to the whole affine  
complexification $X_{\C}=U_\C/ K_\C$ of $X$. This is no longer true
in case (iii).   However in this case, 
and more generally for a noncompact Riemannian
symmetric space of any rank $X=G/K$,
there exists a $G$-invariant domain $\Cr(X)$
of $X_{\C}=G_\C/ K_\C$ containing $X$, the {\it complex crown},  
with the following property (\cite{KSt}, \cite{KS}):

{\em every $\L$-eigenfunction on $X$ admits a holomorphic extension to $\Cr(X)$, and this domain is maximal for this property}.

The objective  of this paper is to obtain analogous theory for the spaces in (iv) above.

\par We note that all spaces in (iii), except for $H^n(\R)$, fall into class (iv) by identifying the
symmetric space $X=G/K$ with the $NA$-part in the Iwasawa decomposition $G=NAK$ of a noncompact
simple Lie group $G$ of real rank one, and by suitably scaling the metric. 

\par Our investigations start with a new model of the crown domain
for the rankone symmetric spaces $X$. We describe $\Cr(X)$ in terms of the Iwasawa coordinates $A$ and $N$ only; 
henceforth we refer to this new model  as the {\em mixed model} of the crown.
In Section 2 we provide the mixed model for the two  basic cases, i.e.
the symmetric spaces associated with the groups $G=\Sl(2,\R)$ and $G=\SU(2,1)$. 

\par Starting from the two basic cases, reduction of symmetry allows to obtain a mixed model for all 
rank one symmetric spaces and motivates a definition of $\Cr(S)$ for the remaining 
spaces in (iv).  This is worked out in Section 3. 
\par Finally, in section 4 we use recent results from \cite{KS} to prove holomorphic extension 
of $\L$-eigenfunctions on $S$ to the crown domain $\Cr(S)$ and establish 
maximality of $\Cr(S)$ with respect to this property.

\section{Mixed model for the crown domain}

The crown domain can  be realized inside the complexification of 
an Iwasawa $AN$-group. Goal of this section is to make 
this explicit for the rank one groups $\Sl(2,\R)$ and 
$\SU(2,1)$.

\subsection{Notation for rank one spaces}

Let $G$ be a connected semi-simple Lie group of 
real rank one. We assume that $G\subset G_\C$ where $G_\C$ is 
the universal complexification of $G$. 
We fix an Iwasawa decomposition $G=NAK$ and form 
the Riemannian symmetric space 
$$X=G/K\, .$$
With $K_\C<G_\C$ the universal complexification of $K$  
we arrive at a totally real embedding 
$$X\hookrightarrow X_\C:= G_\C/ K_\C, \ \ gK\mapsto gK_\C\, .$$
Let $x_0=K\in X$ be a base-point. 

\par Let $\gf, \kf, \af$ and $\nf$ be the Lie algebras 
of $G, K, A $ and $N$. Let $\Sigma^+=\Sigma(\af, \nf)$ be the set of positive 
roots and put 

$$\Omega:=\{Y\in \af\mid  \forall \alpha\in \Sigma^+\ 
|\alpha(Y)|< \pi/2\}\, .$$
Note that $\Omega$ is a symmetric interval in $\af\simeq \R$. 
The crown domain of $X$ is defined as 

\begin{equation} \label{p=e} 
\Cr(X):=G\exp(i\Omega)\cdot x_0\subset X_\C\, . \end{equation}
Let us point out that $\Omega$ is invariant under the Weyl group 
$\W=N_K(A)/Z_K(A)\simeq \Z_2$ and  
that $\exp(i\Omega)$ consists of elliptic elements 
in $G_\C$. We will refer to (\ref{p=e}) as the {\it elliptic model} 
of $\Cr(X)$ (see \cite{KSt} for the basic structure theory in these coordinates).

\par Let us define a domain $\Lambda$ in $\nf$ by 

$$\Lambda:=\{ Y\in \nf\mid \exp(iY)\cdot x_0\subset \Cr(X)\}_0$$
where $\{\cdot\}_0$ refers to the connected component 
of $\{\cdot\}$ which contains $0$.

\par The set $\Lambda$ is explicitly determined in \cite{KO}, Th. 8.11.
Further by \cite{KO}, Th. 8.3: 

\begin{equation} \label{p=u} \Cr(X)=G\exp(i\Lambda)\cdot x_0\, .
\end{equation}
We refer to (\ref{p=u}) as the {\it unipotent model} 
of $\Cr(X)$.

Finally let us mention the fact that $\Cr(X)\subset N_\C A_\C \cdot x_0$
which brings us to the question whether $\Cr(X)$ can be expressed in terms 
of $A, N$, $\Omega$ and $\Lambda$. 
This is indeed the case and will be considered in 
the following two subsections for the groups $\Sl(2,\R)$ and 
$\SU(2,1)$.

\subsection{Mixed model for the upper half plane}

Let 
$$G=\Sl(2,\R) \quad \hbox{and}\quad G_\C= \Sl(2,\C)\, .$$
Our choices of $A, N$ and $K$ are as follows: 

\begin{align*} A &=\left \{ a_t=\begin{pmatrix} t & 0 \\ 0 & 1/t
\end{pmatrix}\mid t>0\right\}\, ,\\
A_\C &=\left \{ a_z=\begin{pmatrix} z & 0 \\ 0 & 1/z
\end{pmatrix}\mid z\in \C^*\right\}\, ,\end{align*}
$$K=\SO(2,\R) \quad \hbox{and} \quad K_\C=\SO(2, \C)\, ,$$ 
and 
\begin{align*} N &=\left \{ n_x=\begin{pmatrix} 1 & x \\ 0 & 1
\end{pmatrix}\mid x\in\R\right\}\, ,\\
N_\C &=\left \{ n_z=\begin{pmatrix} 1 & z \\ 0 & 1
\end{pmatrix}\mid z\in \C\right\}\, . \end{align*}

We will identify $X=G/K$ with the 
upper halfplane ${\bf H}=\{ z\in \C\mid \Im z>0\}$ via the map 

\beq \label{e1} X\to {\bf H}, \ \ gK\mapsto {a i +b \over ci +d} \qquad 
\left(g=\begin{pmatrix} a & b\\ c & d \end{pmatrix}\right)\, .\eeq
Note that $x_0=i$ 
within our identification. 

\par We view $X={\bf H}$ inside of the complex projective 
space $\P^1(\C)=\C\cup\{\infty\}$ and note that 
$\P^1(\C)$ is homogeneous for $G_\C$ with respect 
to the usual fractional linear action: 

$$g(z)= {a z +b \over cz +d}\qquad \left(z\in \P^1(\C), 
g=\begin{pmatrix} a & b\\ c & d \end{pmatrix}\in G_\C\right)\, .$$  

\par We use a more  concrete model for  $X_\C=G_\C/K_\C$, namely 
$$X_\C\to \P^1(\C)\times \P^1(\C)\bs{\diag}, \ \ gK_\C\mapsto \left(g(i), g(-i)\right)  $$
which is a $G_\C$-equivariant diffeomorphism. With this 
identification of $X_\C$ the embedding of (\ref{e1}) becomes 
\begin{equation}\label{e2}  X\hookrightarrow X_\C, \ \ z\mapsto (z,\oline z) \, .\end{equation}

\par We will denote by $\oline X$ the lower half plane and note that 
 the {\it crown domain} for $\Sl(2,\R)$ is given by: 
$$\Cr(X)= X\times \oline X\, .$$

We note that 
$$\Omega=\left\{ \begin{pmatrix} x & 0 \\ 0 & -x\end{pmatrix}\mid 
x\in(-\pi/4, \pi/4)\right\}\, , $$
and 
$$\Lambda=\left\{\begin{pmatrix} 0 & x \\ 0 & 0\end{pmatrix}\mid x\in (-1,1)
\right\}\, .$$

\par With that we come to the mixed model for the crown which combines 
both parameterizations in an unexpected way. 

We let $F:=\{\pm {\bf 1}\}$ be the center of $G$ and note that 

$$N_\C A_\C \cdot x_0= \C \times \C \bs \diag$$
and 
$$N_\C A_\C \cdot x_0 \simeq N_\C A_\C / F\, .$$

\begin{prop} \label{sl} Let $G=\Sl(2,\R)$. Then the map 
$$ NA \times \Omega\times \Lambda\to \Cr(X), \ \ (na, H, Y)\mapsto 
na\exp(iH)\exp(iY)\cdot x_0$$
is an $AN$-equivariant  diffeomorphism. 
\end{prop}

\begin{proof} By the facts listed above we only have to show 
that the map is defined and onto. For that 
we first note: 

$$\exp(i\Lambda)\cdot x_0=\{ ((1+t)i , -(1-t)i)\mid t\in (-1,1)\}$$
and thus 

$$A\exp(i\Lambda)= i\R^+ \times - i\R^+\, .$$
Consequently 

$$A\exp(i\Omega)\exp(i\Lambda)\cdot x_0= \{ (z,w)\in \Cr(X)\mid \arg (w)=
\pi + \arg(z)\}$$
and finally 

$$NA \exp(i\Omega)\exp(i\Lambda)\cdot x_0=\Cr(X)$$
as asserted. 
\end{proof}

\subsection{Mixed model for $\SU(2,1)$}

Let $G=\SU(2,1)$.  We let 
$G$ act on $\P^2(\C)=(\C^{3}\setminus\{0\}/ \sim)$
by projectivized linear transformations. 
We embed $\C^2$ into $\P^2(\C)$ via 
$z\mapsto [z,1]$. Then $G$ preserves the ball 
$$X=\{ z\in \C^2\mid \|z\|_2<1\}\simeq G/K $$
with the maximal compact subgroup $K=S(U(2)\times U(1))$
stabilizing the origin $x_0=0\in X$. 
Note that an element 
$$g=\begin{pmatrix} A & u \\ v^t& \alpha\end{pmatrix}\in G$$
with $u,v\in\C^2$ acts on $z\in X$ by 

$$g(z)= {Az +u \over v^t\cdot z+\alpha}\, .$$ 

Now,  as $X$ is Hermitian, the crown is 
given by the double 
$$\Cr(X)= X\times X$$
but with $X$ embedded in $\Cr(X)$ as $z\mapsto (z,\oline z)$. 
We choose 
$$A=\left \{ a_t:=\begin{pmatrix} 
\cosh t &0 & \sinh t \\ 0 & 1 & 0\\ 
\sinh t &0 & \cosh t \end{pmatrix}\mid t\in \R\right\}\, .$$
Set $Y_t:=\log a_t$ and note that 

$$\Omega=\{Y_t\mid -\pi/4< t < \pi/4\}\, .$$

According to \cite{KO}, Th. 8.11 we have 

\begin{align*}\Lambda=\Big\{ &Z_{a,b}:=\begin{pmatrix}   ib  &  a & -ib \\
-\oline a  &0  & \oline a\\  ib  &  a & -ib\end{pmatrix}
\mid a\in \C, b\in \R; \\
&|a|^2 +|b|<1/2\}\, .\end{align*}

We define now a subset $\D\subset \af \times \nf $ by 

$$\D:=\{ (Y_t, Z_{a,b})\in \af \times \nf \mid 
(1 - 2|a|^2 - 2|b|) \cos (2 t) > (1-\cos (2t)) |a|^2 \}_0\,.$$
If $\pi_\af: \af\times \nf  \to \af$ denotes 
the first coordinate projection and $\pi_\nf$ resp. the second, 
then note the following immediate facts: 
\begin{itemize}
\item $\D\subset \Omega\times \Lambda,$
\item $\pi_\af (\D) = \Omega$, 
\item $\pi_\nf(\D)= \Lambda$. 
\end{itemize}

\begin{prop} \label{su} Let $G=\SU(2,1)$. Then the map 
$$\Phi: NA\times \D \to \Cr(X), \ \ (na, (Y, Z))
\mapsto na \exp (iY) \exp(i Z)\cdot x_0 $$
is a diffeomorphism. 
\end{prop}

\begin{proof} All what we have to show is that the map is 
defined and onto. 
\par To begin with we show that $\Phi$ is defined, i.e.
$\Im \Phi \subset \Cr(X)$. 
Set $n_{a,b}:= \exp(iZ_{a,b})$. Let $M=Z_K(A)$. By $M$-invariance 
it is no loss of generality to assume that $a$ is real. 
Then 
$$n_{a,b} =\begin{pmatrix} 1-b +a^2/2 & i a & b - a^2/2\\
- i a & 1 & i a\\ -b +a^2/2 & i a & 1+b -a^2/2\end{pmatrix}\, .
$$

We have to show that: 

$$a_{i\phi} n_{a,b} (0)\in X \quad{and}\quad \oline {a_{i\phi} n_{a,b} (0)}\
\in X $$
for all $|\phi|<\pi/4$ and  $|a|^2 + |b| <1/2$. 

Now 
$$n_{a,b}(0)={1\over 1+ b -a^2/2}(b-a^2/2, i a)\, .$$
Note that $n_{a,b}(0)\in X$ if and only if
$$(b-a^2/2)^2 + a^2 < (1+ b -a^2/2)^2$$
or 
$$-2b + 2a^2 <1$$
which is the defining condition of $\Lambda$.  
\par 
Applying $a_{i\phi}$ we obtain that  
$$a_{i\phi} n_{a,b}(0)= {\left(\cos\phi {b-a^2/2 \over 1 +b -a^2/2} + i\sin\phi,  {ia 
\over 1 +b -a^2/2}\right)\over \cos\phi + i\sin \phi { b - a^2/2\over 
1+ b -a^2/2}}\, .$$ 

Hence $a_{i\phi} n_{a,b}(0)\in X$ if and only if 
\begin{align*} 
\cos^2\phi + \sin^2 \phi { (b - a^2/2)^2 \over 
(1+ b -a^2/2)^2} > &\cos^2\phi {(b-a^2/2)^2 \over (1 +b -a^2/2)^2}+ \\
&+ \sin^2\phi + {a^2\over (1+b-a^2/2)^2}\end{align*}
or, equivalently, after clearing denominators:   

\begin{align*} 
(1+ b -a^2/2)^2 \cos^2\phi + (b - a^2/2)^2 &\sin^2 \phi 
> (b-a^2/2)^2 \cos^2\phi+ \\
&+ (1+b-a^2/2)^2\sin^2\phi + a^2\, .\end{align*}
Simplifying further we arrive at: 
$$ \cos^2\phi + 2(b-a^2/2) \cos^2\phi >  \sin^2\phi+ 
2(b-a^2/2)\sin^2\phi + a^2$$
and equivalently: 
\begin{equation}\label{in}
(1- 2b - 2a^2) \cos (2\phi)> (1-\cos 2\phi) a^2\, .\end{equation}
But this is the defining condition for $\D$. 

\par To see that the map is onto we observe that 
$\Cr(X) \subset N_\C A_\C \cdot x_0$. Hence there exist a 
domain $\D'\subset \af +\nf$  such that 

$$\Cr(X)=NA \cdot \{\exp(iY)\exp(iX)\mid (Y,X)\in \D'\}\, .$$ 
If $\D'$ is strictly larger then $\D$, then $\D'$ contains 
a boundary point of $\D$. As 
$$\partial \D =\partial \Omega \times\Lambda 
\amalg  \Omega \times \partial \Lambda
\amalg \partial \Omega \times\partial\Lambda, $$
we arrive at a contradiction with (\ref{in}).  
\end{proof}

\section{The crown for homogeneous harmonic spaces}

Let $S$ be a simply connected noncompact homogeneous harmonic space. 
According to \cite{He}, Corollary 1.2, there are the following possibilities for $S$:

\begin{enumerate}
\item $S=\R^n$. 
\item $S$ is the $AN$-part of a noncompact simple Lie group $G$ of real 
rank one. 
\item $S= A \ltimes N $ with $N$ a nilpotent group of Heisenberg-type and 
$A\simeq \R$ acting on $N$ by graduation preserving scalings. 
\end{enumerate}

The case of $S=\R^n$ we will not consider; groups under (i) 
are referred to as symmetric solvable harmonic groups. 
We mention that all spaces in (ii), except for those associated 
to $G=\SO_o(1,n)$, are of the type in (iii). 
Most issues of the Lorentz groups $G=\SO_o(1,n)$ readily reduce to 
$\SO_o(1,2)\simeq \PSl(2,\R)$ where 
comprehensive treatments are available. In fact for 
the real hyperbolic spaces  $X= H^n(\mathbb{R}) =
\SO_o(1,n)/\SO(n)$ we obtain the following result, 
 analogous to Proposition \ref{sl}.

\begin{prop} Let $G=\SO_o(1,n)$ ($n\geq 2$). 
Then the map 
$$ NA \times \Omega\times \Lambda\to \Cr(X), \ \ (na, H, Y)\mapsto 
na\exp(iH)\exp(iY)\cdot x_0$$
is a diffeomorphism. 
\end{prop}
     
We will now focus on type (iii).  
In the sequel we recall some basic facts about $H$-type groups and their 
solvable harmonic extensions. 
We refer to \cite{ROU} for a more comprehensive treatment and references. 
After that we introduce the crown domain for such harmonic extensions.

\subsection{$H$-type Lie algebras and groups}

Let $\nf$ be a real nilpotent Lie algebra of step two (that is, $[\nf,\nf]\neq \{0\}$ and $[\nf,[\nf,\nf]]=\{0\}$), 
equipped with an inner product $\langle\cdot ,\cdot\rangle$ and associated norm $| \cdot |$. Let $\zf$ be the center  
of $\nf$ and $\vf$ its orthogonal complement in $\nf$. Then
$$
\nf=\vf\oplus\zf,\;\;\;\;[\vf,\zf]=0,\;\;\;[\nf,\nf]=[\vf,\vf]\subset \zf. 
$$
For $Z\in\zf$ let $J_Z:\vf\rightarrow \vf$ be the linear map defined by 
\beq
\notag
\langle J_Z V,V'\rangle=\langle Z, [V,V']\rangle,\;\;\;\forall V,V'\in\vf.
\eeq
Then $\nf$ is called a {\em Heisenberg type} algebra (or {\em H-type} algebra, for short) if      
\beq
\label{HTYPE1}
J_Z^2=-|Z|^2 \id_{\vf} \qquad  (Z  \in\mathfrak{z}).
\eeq
A connected and simply connected Lie group $N$ is called an H-type group if its Lie algebra $\nf=\mbox{Lie}(N)$ is an H-type algebra, 
see \cite{KAPLAN}.  

We let $p=\mbox{dim}\,\mathfrak{v}$,  $q=\mbox{dim}\,\mathfrak{z}(\geq 1)$. Condition (\ref{HTYPE1}) implies that $p$ is even
and $[\vf,\vf]=\zf$. 

Moreover,  (\ref{HTYPE1}) implies that the 
map $Z\rightarrow J_Z$ extends to a representation of the real Clifford algebra $\mbox{Cl}(\zf)\cong \mbox{Cl}_q$ on $\vf$. This procedure can be
reversed and  yields a general method for constructing $H$-type algebras.

Since $\nf$ is nilpotent, the exponential map
$\exp:\nf\mapsto N$ is a diffeomorphism. The Campbell-Hausdorff formula implies the following product law in $N$:
$$
\exp X\cdot\exp X'=\exp\left(X+X'+\tfrac{1}{2}[X,X']\right),\;\;\;\;\forall X,X'\in\nf.
$$
This is sometimes written as
$$
(V,Z)\cdot (V',Z')=\left( V+V',Z+Z'+\tfrac{1}{2}[V,V']\right), 
$$
 using the exponential chart to parametrize the elements $n=\exp(V+Z)$ by the couples $(V,Z)\in\vf\oplus\zf=\nf$. 

\subsubsection{Reduction theory} 
We conclude this section with reduction theory for $H$-type Lie algebras to 
Heisenberg algebras. 

\par Let $\zf_1=\R Z_1$ be a one-dimensional subspace of $\zf$ and $\zf_1^\perp$ its 
orthogonal complement in $\zf$.  We assume that $|Z_1|=1$ and set $J_1:= J_{Z_1}$. 
We form the 
quotient algebra 

$$\nf_1:=\nf/ \zf_1^\perp\, $$
and record that $\nf_1$ is two-step nilpotent. Let $p_1:\zf\to \zf_1$ be the orthogonal 
projection. If we identify $\nf_1$ with the vector space $\vf \oplus \zf_1$ via
the linear map 

$$\nf_1 \to \vf\oplus \zf_1, \ \ (V,Z)+ \zf_1^\perp \mapsto (V, p_1(Z))\, ,$$
then the bracket in $\nf_1$ becomes in the new coordinates 

$$[ (V, cZ_1), (V' , c' Z_1)] =  (0, \langle Z_1, [V, V']\rangle Z_1 )$$
where $c, c'\in \R$ and $V, V'\in \vf$.
Since $\langle Z_1, [V, V'] \rangle = \langle J_1 V, V'\rangle $ we see that $J_1$ determines 
a Lie algebra automorphism of $\nf_1$ and thus $\nf_1$ is isomorphic
to the $p+1$-dimensional Heisenberg algebra.

\subsection{Harmonic solvable extensions of $H$-type groups}

Let $\nf$ be an H-type algebra with associated H-type group $N$. Let $\af$ be a one-dimensional Lie algebra with an inner product. Write
$\af=\R H$, where $H$ is a unit vector in $\af$. Let $A=\exp\af$ be a one-dimensional Lie group with Lie algebra $\af$ and isomorphic to 
$\mathbb{R}^+$ (the multiplicative group of positive real numbers).  Let the elements $a_t=\exp(tH)\in A$ act on $N$ by the dilations
$(V,Z)\rightarrow (e^{t/2}V,e^t Z)$  for $t\in\mathbb{R}$, and let $S$ be the associated semidirect product of $N$ and $A$:
$$
S=NA=N\rtimes A.
$$

The action of $A$ on $N$ becomes the inner automorphism
\beq
\label{AN1}
a_t\exp(V+Z)a_t^{-1}= \exp\left( e^{t/2}V+e^t Z\right),
\eeq
and the product in $S$ is given by
$$
\exp(V+Z)a_t\,\exp(V'+Z')a_{t'}=\exp(V+Z)\exp(e^{t/2}V'+e^tZ') a_{t+t'}.
$$

$S$ is a connected and simply connected Lie group with Lie algebra
$$
\sf=\nf\oplus \af=\vf\oplus\zf\oplus \af 
$$
and Lie bracket defined by linearity and the requirement that 
\beq
\label{CR1}
[H,V]=\frac{1}{2}V,\;\;\;\;\;\;\;[H,Z]=Z,\;\;\;\;\;\;\;\forall V\in\mathfrak{v},\;\forall Z\in\mathfrak{z}.
\eeq

The map $(V,Z,tH)\rightarrow \exp(V+Z)\exp(tH)$ is a diffeomorphism
of $\sf$ onto $S$. If we 
parametrize the elements $na= \exp(V+Z)\exp(tH)\in NA$ by the triples $(V,Z,t)\in\vf\times\zf\times\R$, then the 
product law reads 
$$
(V,Z,t)\cdot (V',Z',t')=\left(V+e^{t/2}V',Z+e^tZ'+\tfrac{1}{2}e^{t/2}[V,V'], t+t'\right),
$$
for all $V,V'\in\vf$, $Z,Z'\in\zf$, $t,t'\in\R$.
For $n=(V,Z,0)\in N$ and $a_t=(0,0,t)\in A$ we consistently get $na_t=(V,Z,t)$.

We extend the inner products on $\nf$ and $\af$ to an inner product $\langle \cdot,\cdot\rangle$ on $\sf$ by 
linearity and the requirement that $\nf$ be orthogonal to $\af$. 
The left-invariant Riemannian metric on $S$ defined by this inner product 
turns $S$ into a harmonic solvable group \cite{DR1}.

\subsection{The complexification and the crown}

Let $N_\C$ be the simply connected Lie group with Lie algebra $\nf_\C$ and set 
$A_\C =\C^*$. 
 Then $A_\C$ acts on $N_\C$ by $t\cdot (V, Z):=(tV, t^2V)$ for 
$t\in \C^*$ and $(V,Z)\in N_\C$. 

We define a complexification $S_{\C}$ of $S$ by 

$$S_\C = N_\C \rtimes A_\C\, . $$ 
For $z\in \C$ we often set $a_z:=\exp(zH)$ and note that $a_z$ corresponds to $e^{z/2}\in \C^*$.  
In particular 
\beq
\label{KERNEL}
\{z\in\mathbb{C}:\;\;\exp(zH)=e\}\subset 4\pi i\mathbb{Z}.
\eeq
It follows that the exponential map 
$\exp:\mathfrak{a}_{\mathbb{C}}\rightarrow A_{\mathbb{C}}$ is certainly
injective if restricted to $\mathfrak{a}\oplus iH(-2\pi,2\pi]$.

Motivated by our discussion of $\SU(2,1)$ and the discussed 
$\SU(n,1)$-reduction, 
we define the following sets for a more general harmonic $AN$-group. 

\beq
\label{OME1}
\Omega=\{tH\in\mathfrak{a}:\;\;|t|<\tfrac{\pi}{2}\},
\eeq
\beq
\label{LAM1}
\Lambda=\{(V,Z)\in\mathfrak{n}:\;\;\tfrac{1}{2}|V|^2+|Z|<1\},
\eeq
\beq
\label{D1}
\D=\left\{(V,Z,t)\in\sf:\;\cos t( 1-\tfrac{1}{2}|V|^2-|Z|) >\tfrac{1}{4} (1-\cos t)|V|^2\right\}_0,
\eeq
\beq
\label{D2}
D=\{\exp(itH)\exp(iV+iZ):\;\;(V,Z,t)\in \D\}\subset S_{\C}.
\eeq
Here as usual $\{\cdot \}_0$ denotes the connected component of $\{\cdot\}$
containing 0, and  we write $(V,Z,t)$ for the element $V+Z+tH$ of $\sf$.

The following result is immediate from the definition of $\D$. 

\begin{lem}  \label{PROPO1}    Let $\pi_{\nf}$ (resp. $\pi_{\af}$) denote the projection onto the first (resp. second) factor in 
$\sf= \nf\times\af$. Then

\begin{itemize}
\item $\D\subset \Lambda\times\Omega$, 
\item $\pi_{\nf} (\D) = \Lambda$, 
\item $\pi_{\af}(\D)= \Omega$. 
\end{itemize}
\end{lem}

The set  $\D$ is the interior of the closed hypersurface in $\sf$ defined by the equation 
\beq
\label{IPE1}
\cos t( 1-\tfrac{1}{2}|V|^2-|Z|)=\tfrac{1}{4}(1-\cos t)|V|^2 \;\;\;\;(-\tfrac{\pi}{2}\leq t\leq\tfrac{\pi}{2}).
\eeq

\medskip

\noindent This can be rewritten as $\;\cos t(1-|V|^2/4-|Z|)=|V|^2/4,\;$ or also as
$$
1-\tan^2\tfrac{t}{2}=\tfrac{|V|^2}{2(1-|Z|)},
$$
i.e.,
$$
|t|=2\arctan\sqrt{1-\tfrac{|V|^2}{2(1-|Z|)}}.
$$
\smallskip
A picture of this hypersurface in $\R^3=\{(V,Z,t)\}$ for the case $\zf\simeq \R$, $\vf\simeq \R^2$ (with one coordinate suppressed), i.e., 
for $\SU(2,1)$, is given below.  Here $|t|\leq \pi/2$, $|Z|\leq 1$, and $|V|\leq \sqrt{2}$

\

A plot of the surface (\ref{IPE1}) in $\R^3=\{(V,Z,t)\}$.
\qquad

\vskip20pt

\begin{center}
\scalebox{1}
{\includegraphics{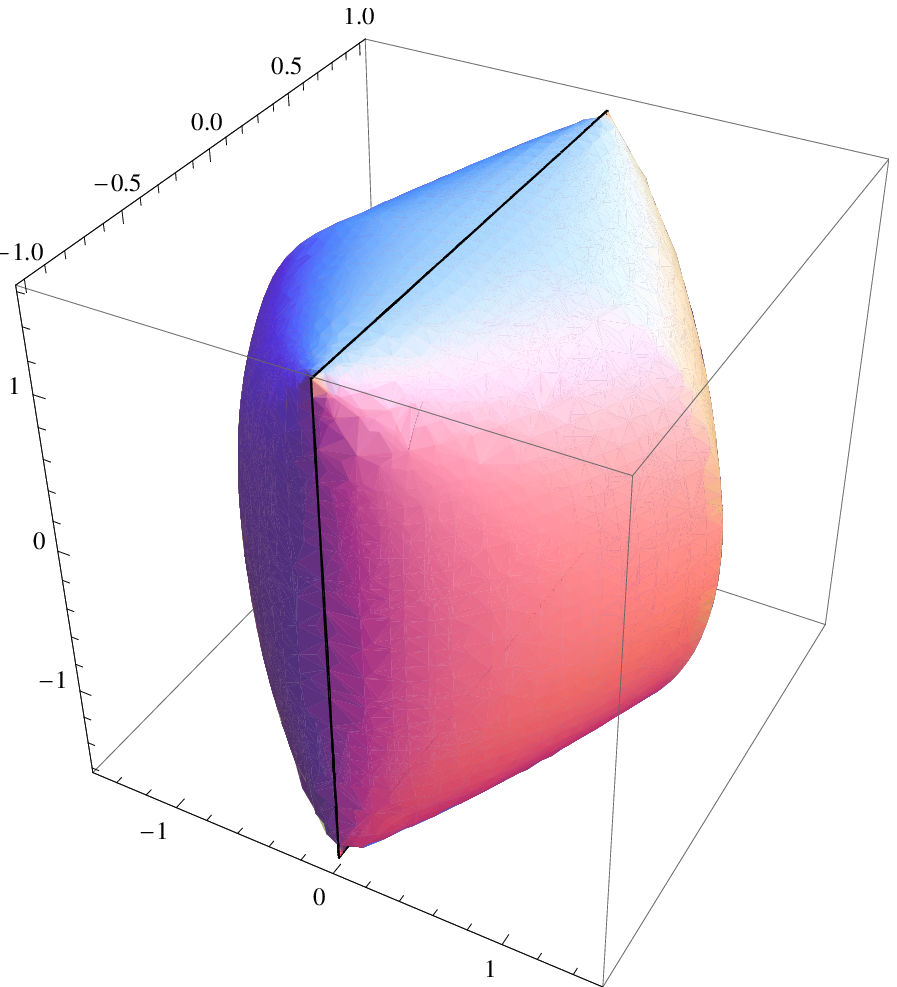}}
\end{center}

\qquad

\begin{df}
The complex crown of $S$ will be defined as the following subset of $S_{\mathbb{C}}$:
\beq
\label{CR44}
\mbox{\em Cr}(S)=NAD
\subset 
NA\exp(i\Omega)\exp(i\Lambda)\subset S_{\C}.
\eeq
\end{df}

It is easy to see that $\mbox{Cr}(S)$ is open and simply connected in $S_{\C}$.

We conclude with the proof of the mixed model of the crown for rank one symmetric spaces.

\begin{thm} \label{th=1} Let $X=G/K$ be a non-compact Riemannian symmetric space of 
rank one, $X\neq H^n(\mathbb{R})$. Then, with $S=NA$, 
$$\Cr(X)=\Cr(S) .$$
\end{thm}
 
\begin{proof} We first show that $\Cr(S)\subset \Cr(X)$.
 We have to show that 
$\exp(itH)\exp(iY)\cdot x_0\subset \Cr(X)$ for all 
$(Y,t)\in \D$. We apply $M=Z_K(A)$ and the assertion is reduced to $G=\SU(2,1)$,
where it was shown in Proposition \ref{su} above.  

\par In order to conclude  the proof of the theorem, it is enough to show that 
the elements which are in the boundary of $D$ do not lie in  $\Cr(X)$, i.e.,
$$
\{\exp(itH)\exp(iV+iZ):\;\;(V,Z,t)\in \partial\D\} \cap  \Cr(X) =\emptyset.
$$
This again reduces to  $G=\SU(2,1)$, and is easily verified.
\end{proof}

\section{Geometric Analysis}

In this section we will show that eigenfunctions of the Laplace-Beltrami-Operator 
on $S$ extend holomorphically to $\Cr(S)$ and that $\Cr(S)$ is maximal 
with respect to this property. 

\subsection{Holomorphic extension of eigenfunctions}

\par For $z\in D$ we consider the following totally real embedding 
of $S$ into $\Cr(S)$ given by 

$$S\hookrightarrow \Cr(S), \ \ s\mapsto sz\, .$$
Now Let $\L$ be the Laplace-Beltrami operator on $S$, explicitly 
given by, see \cite{ROU} 

$$\L:=\sum_{j=1}^p V_j^2 +\sum_{i=1}^q Z_i^2 + H^2 - 2\rho H$$
where the $(V_j)_j$ and $(Z_i)_i$ form an orthonormal basis of 
$\vf$ and $\zf$ respectively.  Here we consider elements $X\in \sf$ as left-invariant 
vector fields on $S$. Hence it is clear that $\L$ extends to a left $S_\C$-invariant 
holomorphic differential operator on $S_\C$ which we denote by $\L_\C$. 
Now if $M\subset S_\C$ is a totally real analytic submanifold, then we can restrict 
$\L_\C$ to $M$, in symbols $\L_M$,  in the following way: if $f$ is a real analytic function 
near $m\in M$ and $f_\C$ a holomorphic extension of $f$ in a complex neighborhood 
of $m$ in $S_\C$, then set:

$$(\L_M f)(m):= (\L_\C f_\C)(m)\, .$$    
Then: 

\begin{prop} For all $z\in D$ the restriction $\L_{Sz}$ is elliptic. 
\end{prop}

\begin{proof} To illustrate what is going on we first give a proof for those $S$ related to 
$G=\Sl(2,\R)$. Here we have $D:=\exp(i\Omega)\exp(i\Lambda)$ and 
$\L= V^2 + H^2 -{1\over 2} H$. 
Let $z=\exp(itH)\exp(ixV)\in D$. Using $[H,V]=V$ we get
\begin{align*} \Ad(z) H  & = H - i e^{it}x V \\ 
\Ad(z) V & = e^{it} V\, .\end{align*}
It follows that the leading symbol, or principal part, of $\L_{Sz}$ is given by:

$$[\L_{Sz}]_{\rm prin}= (H- i e^{it}x V)^2 + e^{2it}V^2\, .$$
Let us verify that $\L_{Sz}$ elliptic. The associated quadratic form is given by 
the matrix 

$$L(z):=\begin{pmatrix}  1 & - i x e^{it} \\ - ix e^{it} & e^{2it}(1-x^2) \end{pmatrix}\, .$$   
We have to show that $\langle  L(z)\xi, \xi\rangle =0$ has no solution 
for $\xi=(\xi_1, \xi_2)\in \R^2\backslash \{0\}$. 
We look at 

$$\xi_1^2 - 2ix e^{it} \xi_1 \xi_2  + e^{2it} (1-x^2) \xi_2^2 =0\, .$$
Now $\xi_2=0$ is readily excluded and we remain with the quadric 
$$\xi^2 - 2ix e^{it} \xi + e^{2it} (1-x^2)=0 $$
whose solutions are 

\begin{align*} \lambda_{1,2} & = ixe^{it} \pm \sqrt{-x^2 e^{2it}- (1-x^2)e^{2it}}\\
& =  ie^{it}  ( x  \pm  1 ).\end{align*} 
These are never real in the domain $-\tfrac{\pi}{2}<t<\tfrac{\pi}{2}$ and
$-1<x<1$. The same proof works for those $S$ related to $\SO_o(1,n)$ , $n\geq 2$.

\par To put the computation above in a more abstract framework: it is to show here that for $z\in D$ the operator $L(z): \sf_\C \to \sf_\C$ 
defined by 

\begin{equation} \label{e12} L(z):= \Ad(z)^t \Ad (z) \end{equation}
is elliptic in the sense that $\langle L(z)\xi, \xi\rangle =0 $ for $\xi \in \sf$ implies $\xi=0$.

So for the sequel we assume that $\nf$ is $H$-type, i.e. $\zf\neq \{0\}$. We will 
reduce to the case where $N$ is a Heisenberg group, i.e. $S$ is related to 
$\SU(1,n)$.  We use the already introduced technique of reduction to Heisenberg groups.
So let $Z_1\in \zf$ be a normalized element and $\nf_1=\nf/ \zf_1^\perp$ as before.
We let $S_1$ be the harmonic group associated to $N_1$ and note that 
there is a natural group homomorphism $S\to S_1$ which extends to a holomorphic 
map $S_\C \to (S_1)_\C$ which  maps $\Cr(S)$ onto $\Cr(S_1)$. 
Now the assertion is true for $\Cr(S_1)$ in view of \cite{KS} (proof of Th. 3.2) 
and Theorem \ref{th=1}.
Since (\ref{e12}) is true for $S$ if it is true for all choices of $S_1$, 
the assertion follows.
\end{proof} 

As a consequence of this fact, we obtain as in \cite{KS} that:

\begin{thm} Every $\L$-eigenfunction on $S$ extends to a holomorphic function 
on $\Cr(S)$.
\end{thm}

\begin{proof} (analogous to the proof of Th. 3.2 of  \cite{KS}). Let $f$ be an $\L$-eigenfunction on $S$.
As $\L$ is elliptic, the regularity theorem 
for elliptic differential operators implies that 
$f$ is an analytic function. Hence $f$ extends to some holomorphic 
function in a neighborhood of $S$ in $S_\C$. 
 
\par As $\L$ is $S$-invariant, we may assume that this 
neighborhood is $S$ invariant.  Let $0\leq t\leq 1$ and define $\D_t:=t\D$
and correspondingly $D_t$. 
We have shown that $f$ extends holomorphically to a domain 
$SD_t\subset S_\C$ for some $0<t\leq 1$.

\par If $t=1$, we are finished. Otherwise we find 
a $(Y,r) \in \partial \D_t$  such that 
$f$ does not extend beyond $z=s\exp(irH)\exp(iY) $ for some 
$s\in S$. By $S$-invariance we may assume that $s=1$.
Set $z:=\exp(irH)\exp(iY)$. By our previous Proposition $\L_{Sz}$
is elliptic. 
Now it comes down to choose appropriate local 
coordinates to see that $f$ extends holomorphically on a 
complex cone based at $z$. One has to verify condition 
(9.4.16) in \cite{H}, Cor. 9.4.9 so that 
\cite{H}, Cor. 9.4.9, applies and one concludes that 
$f$ is holomorphic near $z$ -- see \cite{KS}, p. 837-838, for the details. This is a contradiction 
and the theorem follows.
\end{proof}

\subsection{Maximality of $\Cr(S)$}

We begin with a collection of  some facts about Poisson kernels on $S$. 
Let us denote by $s:S \to S$ the geodesic symmetry, centered at the identity.
Every $z\in S_\C$ can be uniquely written as $z=n(z) a(z) $ with $n\in N_\C$ and $a\in A_\C$. 
As $\Cr(S)$ is simply connected,  we obtain for every $\lambda\in \af_\C^*$ a holomorphic map 

$$a^\lambda: \Cr(S) \to \C, \ z\mapsto e^{\lambda\log a(z)}\, .$$
The function $P_\lambda:= a^\lambda\circ s$ on $S$ is referred to as {\it Poisson kernel 
on $S$ with parameter $\lambda$}. 
We note that both $a^\lambda$ and $P_\lambda$ are $\L$-eigenfunctions \cite{D,ADY}. 
We shall say in the sequel that $\lambda\in \af_\C^*$, resp. $P_\lambda$,  is {\it positive} 
if $\Re\lambda$ is a positive multiple of the element $\beta\in \af_\C^*$ defined by $\beta(H)=1$.

\begin{thm} Let $\lambda\in \af_\C^*$ be positive. Then $\Cr(S)$ is the largest $S$-invariant 
domain in $S_\C$ containing $S$ to which $P_\lambda$ extends holomorphically. 
\end{thm} 
\begin{proof} First the theorem is true for symmetric $S$. To be more precise, given $z\in \partial \Cr(S)$, then 
it was shown in \cite{KOS}, Th. 5.1 (and corrigendum in \cite{KO}, Remark 4.8) that there 
exists an $s\in S$ such that the basic 
spherical function $\phi_0$ on $S\simeq X$ does not extend beyond $sz$. 
In view of the integral formulas for holomorphically extended spherical functions
(see \cite{KSt1}, Th. 4.2) it follows that $\Cr(S)$ is the maximal $S$-invariant 
domain in $\Cr(S)$ to which any  positive $P_\lambda$ 
extends holomorphically.   
\par We apply $\SU(2,1)$-reduction (see \cite{CDKR1}, section 2) 
and the fact that Poisson-kernels ``restrict'', i.e. if $z\in \Cr(S)$, then we can put $z \in \Cr(S_1)$ with 
$\Cr(S_1)$ an $\SU(2,1)$-crown contained in $\Cr(S)$ such that the restriction of the Poisson kernel 
$P_\lambda:=P_\lambda|_{S_1}$ is a positive Poisson kernel on $S_1$. This is clear from the explicit
formula for $P_{\lambda}$ (see \cite{ADY}, formula (2.35)). Thus the situation is reduced to the 
symmetric case and the theorem proved.
\end{proof}

\begin{cor} \label{th=max} $\Cr(S)$ is the largest $S$-invariant domain in $S_\C$ which contains $S$ with the  
property that every $\L$-eigenfunction on $S$ extends to a holomorphic function on $\Cr(S)$. 
\end{cor}

\begin{cor} The geodesic symmetry extends to a holomorphic involutive map $s: \Cr(S)\to \Cr(S)$. 
\end{cor}

\end{document}